\documentclass[12pt,reqno]{amsart}
\usepackage[headings]{fullpage}
\usepackage{amssymb,amsmath,amscd,bbm,tikz-cd}
\usepackage[all,cmtip]{xy}
\usepackage{url}
\usepackage[bookmarks=true,%
    colorlinks=true,%
    linkcolor=blue,
    citecolor=blue,%
    filecolor=blue,%
    menucolor=blue,%
    urlcolor=blue,%
    breaklinks=true]{hyperref}
\usepackage{slashed}
\usepackage{listings}
\usepackage{verbatim}
\usepackage{mathtools}
\usepackage[normalem]{ulem}   
\usepackage{graphicx}
\usepackage{texdraw}
\usepackage{caption}
\graphicspath{{figures/}}

\newtheorem{theorem}{Theorem}[section]
\theoremstyle{definition}

\newtheorem{lemma}[theorem]{Lemma}

\newtheorem{remark}[theorem]{Remark}
\newtheorem{corollary}[theorem]{Corollary}
\newtheorem{conjecture}[theorem]{Conjecture}

\newtheorem*{question*}{Question}

\def\BC{\mathbbm C}

\def\calR{\mathcal R}

\def\calG{\mathcal G}
\def\calX{\mathcal X}

\def\SL{\mathrm{SL}}

\def\tr{\mathrm{tr}\,}

\begin{document}
	
\title[ ]{Knots with large character varieties}

\author{Philip Choi}
\address{Quantum Computing Research Center\\
	Chonnam National University \\
	Gwangju, South Korea 
}
\email{philip94@snu.ac.kr}

\author{Joan Porti}
\address{Departament de Matemàtiques\\
	Universitat Autònoma de Barcelona and Centre de Recerca Matemàtica (CRM)\\
	Cerdanyola del Vallès, Spain
    \newline{\tt \url{https://mat.uab.cat/~porti/}}
}
\email{joan.porti@uab.cat}

\author{Seokbeom Yoon}
\address{Department of Mathematics \\
	Chonnam National University \\
	Gwangju, South Korea 
	\newline{\tt \url{https://sites.google.com/view/seokbeom}}
}
\email{sbyoon15@gmail.com}


\date{\today}

\begin{abstract}
    We study knots whose $\mathrm{SL}_2(\mathbb{C})$-character varieties have a component of dimension greater than one. We call such knots $\mathcal{X}$-large and introduce two diagrammatic constructions that produce $\mathcal{X}$-large knots. The first construction uses split link diagrams and rational tangle replacements, providing a topological explanation for most $\mathcal{X}$-large knots observed in knot tables. The second construction is based on braids and orientation-reversing involutions, and is motivated by a detailed analysis of the knot $10_{123}$, also known as the Turk's head knot $Th(3,5)$. In particular, this approach applies to Turk’s head knots $Th(p,q)$ with $p$ and $q$ odd, leading us to conjecture that all such knots are $\mathcal{X}$-large. In doing so, we also present a non-orientable analogue of Thurston's theorem giving a lower bound on the dimension of character varieties of non-orientable 3-manifolds.
\end{abstract}

\maketitle
{
\footnotesize
\tableofcontents
}

\section{Introduction}

Let $K$ be a knot in $S^3$. Here and throughout the paper, we distinguish knots from links: a knot has a single component, whereas a link has more than one component. We denote by $\calG(K):=\pi_1(S^3 \setminus K)$ the knot group of $K$, and by $\calX(K)$ the $\SL_2(\BC)$-character variety of $\calG(K)$,
$$\calX(K) := \mathrm{Hom}(\calG(K), \SL_2(\BC))/\!/_{\SL_2(\BC)}.$$
Character varieties of knot groups have been extensively studied in the literature. However, most existing results focus on the case in which $\calX(K)$ has dimension one. This is partly because, when $K$ is hyperbolic, the geometric representation lies in a one-dimensional component of $\calX(K)$, and partly because many knots (particularly those with few crossings) appear to have one-dimensional character varieties.

In this paper, we study knots $K$ for which $\dim \calX(K)>1$. Here, $\dim$ denotes the (complex) dimension of a component of maximal dimension. We call such knots \emph{$\calX$-large}. It is well known that if $K$ is $\calX$-large, then $K$ is large, meaning that its complement has a closed essential surface \cite{cooper1994plane}. However, the converse does not hold: there exist large knots that are not $\calX$-large, such as $8_{16}$ and $8_{17}$ \cite{knotinfo}. According to computational data in \cite{diagramsite}, obtained in joint work of the first author with Seonhwa Kim on parabolic representations, there are only three candidates in Rolfsen’s table with at most 10 crossings that can be $\mathcal{X}$-large: $10_{98}$, $10_{99}$, and $10_{123}$ (see Appendix~\ref{sec.app} for details). One motivation of this paper is to determine whether these knots are indeed $\calX$-large and, if so, to provide a topological explanation for this phenomenon, rather than relying on brute-force computations of character varieties. To the author’s knowledge, existing related results (for instance, \cite{cooper1996remarks, paoluzzi2013non,chen2025high}) do not apply to these knots.

A main goal of the paper is to introduce two new diagrammatic constructions that produce ${\calX}$-large knots. 

The first construction uses split link diagrams and 2-tangle replacements. We present this construction explicitly in Section~\ref{sec.generalization} (Theorem~\ref{thm.replace}), after explaining the key idea using the knot $10_{98}$ as a toy example in Section~\ref{sec.toyexample}. This construction is simple yet powerful: it produces infinitely many $\calX$-large knots from a fixed link diagram. In fact, most $\calX$-large knots observed in our computation arise from this construction. For example, it provides a diagrammatic explanation of why $10_{98}$ and $10_{99}$ are $\mathcal{X}$-large.
Unfortunately, at present we do not know of a similar explanation for $10_{123}$; however, our second construction applies to $10_{123}$.

The second construction uses braids and orientation-reversing involutions. It was motivated by the observation that the knot $10_{123}$ admits a braid presentation that decomposes into two parts, one of which is conjugate to the other by an orientation-reversing involution (see Figure~\ref{fig.10123}). We present a detailed analysis of $10_{123}$ in Section~\ref{sec.10123} and then generalize this construction in  Section~\ref{sec.sym} (Theorem~\ref{thm.sym}). This construction is applied in particular to Turk's head knots $Th(p,q)$ with $p$ and $q$ odd, which leads us to conjecture that all such Turk's head knots are $\calX$-large (Conjecture~\ref{conj.turk}).
We also present a variant of Theorem~\ref{thm.sym} that applies to strongly positive amphichiral knots (Theorem~\ref{thm.spak}).

The proof of Theorem~\ref{thm.sym} requires a non-orientable analogue of Thurston's theorem giving a lower bound on the dimension of character varieties \cite[Theorem 5.6]{ThurstonNotes}. Roughly speaking, Thurston's theorem asserts that under mild assumptions,  a (connected) orientable 3-manifold with $k$ boundary tori has a character variety of dimension at least $k$. For Theorem~\ref{thm.sym}, we require a corresponding statement for non-orientable 3-manifolds with Klein bottle boundary.
To the best of the authors’ knowledge, such a non-orientable version does not appear elsewhere in the literature. We therefore postpone the proof of Theorem~\ref{thm.sym} to Section~\ref{sec.nonori}, where we formulate and prove the dimension estimation in the non-orientable setting (Theorem~\ref{Theorem:DimensionEstimate}).

Using our constructions, we are able to account for all $\mathcal{X}$-large knots observed in our computations. In particular, as mentioned above, we confirm that $10_{98}$, $10_{99}$, and $10_{123}$ are $\calX$-large. Consequently, we obtain the following statement:
\begin{center}
Up to 10 crossings, the only $\mathcal{X}$-large knots are $10_{98}$, $10_{99}$, and $10_{123}$.
\end{center}
This naturally leads to the question of whether our constructions capture all $\mathcal{X}$-large knots up to a certain crossing number. However, we do not attempt to address this question in this paper, as determining the list of $\mathcal{X}$-large knots requires computing full character varieties and their dimensions, which is extremely difficult even for knots with 11 crossings.

\section{Rational tangle replacements}
\label{sec.tangle}

In this section, we present rational tangle replacements that give rise to $\mathcal{X}$-large knots.
We start with the knot $10_{98}$ as a toy example and then  explain the construction in generality.
\subsection{Toy example: the knot $10_{98}$}  
\label{sec.toyexample}

Let $L$ be the split link consisting of the knot $3_1$ and the unknot. Here, split means that there exists  a 2-sphere separating the unknot from $3_1$.
We fix a diagram of $L$ with a crossing $c$ and two Wirtinger generators $g$ and $h$  around $c$ shown as in Figure~\ref{fig.1098}~(left). Note that the unknot component is colored red.

\begin{figure}[!htbp]
    \centering
\begingroup%
  \makeatletter%
  \providecommand\color[2][]{%
    \errmessage{(Inkscape) Color is used for the text in Inkscape, but the package 'color.sty' is not loaded}%
    \renewcommand\color[2][]{}%
  }%
  \providecommand\transparent[1]{%
    \errmessage{(Inkscape) Transparency is used (non-zero) for the text in Inkscape, but the package 'transparent.sty' is not loaded}%
    \renewcommand\transparent[1]{}%
  }%
  \providecommand\rotatebox[2]{#2}%
  \newcommand*\fsize{\dimexpr\f@size pt\relax}%
  \newcommand*\lineheight[1]{\fontsize{\fsize}{#1\fsize}\selectfont}%
  \ifx\svgwidth\undefined%
    \setlength{\unitlength}{343.55366144bp}%
    \ifx\svgscale\undefined%
      \relax%
    \else%
      \setlength{\unitlength}{\unitlength * \real{\svgscale}}%
    \fi%
  \else%
    \setlength{\unitlength}{\svgwidth}%
  \fi%
  \global\let\svgwidth\undefined%
  \global\let\svgscale\undefined%
  \makeatother%
  \begin{picture}(1,0.37086103)%
    \lineheight{1}%
    \setlength\tabcolsep{0pt}%
    \put(0,0){\includegraphics[width=\unitlength,page=1]{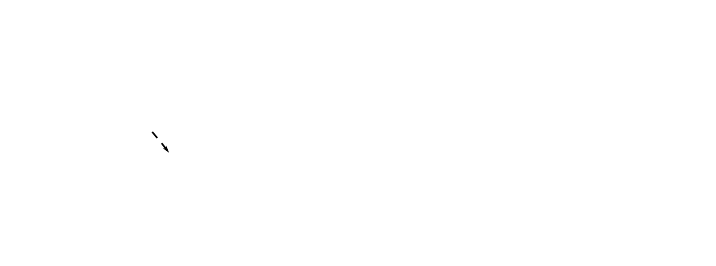}}%
    \put(0.24161356,0.15284313){\makebox(0,0)[lt]{\lineheight{1.25}\smash{\begin{tabular}[t]{l}$g$\end{tabular}}}}%
    \put(0,0){\includegraphics[width=\unitlength,page=2]{1098.pdf}}%
    \put(0.1333721,0.16525682){\makebox(0,0)[lt]{\lineheight{1.25}\smash{\begin{tabular}[t]{l}$g$\end{tabular}}}}%
    \put(0,0){\includegraphics[width=\unitlength,page=3]{1098.pdf}}%
    \put(0.1333721,0.16525682){\makebox(0,0)[lt]{\lineheight{1.25}\smash{\begin{tabular}[t]{l}$g$\end{tabular}}}}%
    \put(0.14497846,0.19801344){\makebox(0,0)[lt]{\lineheight{1.25}\smash{\begin{tabular}[t]{l}$h$\end{tabular}}}}%
    \put(0.20166947,0.13212597){\makebox(0,0)[lt]{\lineheight{1.25}\smash{\begin{tabular}[t]{l}$c$\end{tabular}}}}%
    \put(0.21274523,0.09775105){\makebox(0,0)[lt]{\lineheight{1.25}\smash{\begin{tabular}[t]{l}$g^{-1}hg$\end{tabular}}}}%
    \put(0,0){\includegraphics[width=\unitlength,page=4]{1098.pdf}}%
  \end{picture}%
\endgroup%

    \caption{The split link $3_1 \sqcup O$ and the knot $10_{98}$.}
    \label{fig.1098}
\end{figure}

Let $\rho$ be a representation of $\calG(L)=\pi_1(S^3 \setminus L)$ whose restriction to $\calG(3_1)$ is irreducible and such that $G=\rho(g)$ and $H=\rho(h)$ are of the form
\begin{equation}
\label{eqn.normalization}
    G = \begin{pmatrix} m & 1 \\ 0 & m^{-1} \end{pmatrix}, \quad
    H = \begin{pmatrix} m & 0 \\ u & m^{-1} \end{pmatrix} \,.
\end{equation}
Such $\rho$ exists for any non-zero complex numbers $m$ and $u$, since $h$ is a meridian of the trivial component of the split link $L$, and $g$ is a meridian of the other non-trivial component.

We assume that $(u,m)$ is a solution to the Riley polynomial of $3_1$ 
(\cite{Riley1984}): 
\begin{equation}
\label{eqn.R} u+m^2-1+m^{-2}=0.    
\end{equation}
This means that the matrices $G$ and $H$ correspond to a non-abelian representation of $3_1$.  More precisely, the knot group of $3_1$ has two generators $x$ and $y$ with one relator $xyx=yxy$, and the assignment $x \mapsto G$ and $y \mapsto H$ defines a representation, if Equation~\eqref{eqn.R} is satisfied. Put differently, Equation~\eqref{eqn.R} implies that $GHG = HGH$.

On the other hand, a small circle around the crossing $c$ meets the diagram of $L$ at four points, corresponding to four Wirtinger generators (two of which coincide).
These generators are mapped to $G,H,G$, and $G^{-1}HG$, under the representation $\rho$, as in Figure~\ref{fig.trefoil}~(left).
Since the matrices $G$ and $H$ satisfy $GHG=HGH$, it follows that $HGH^{-1}=G^{-1}HG$ and $HGHG^{-1}H^{-1}=G$. Hence, one may replace the crossing $c$ with the rational $2$-tangle shown in Figure~\ref{fig.trefoil}~(right), keeping the outside of the circle unchanged. This operation yields the knot $10_{98}$ (see Figure~\ref{fig.1098}) and shows that the representation $\rho$ of $\calG(L)$ induces a representation of $\calG(10_{98})$.

\begin{figure}[!htbp]
    \centering
    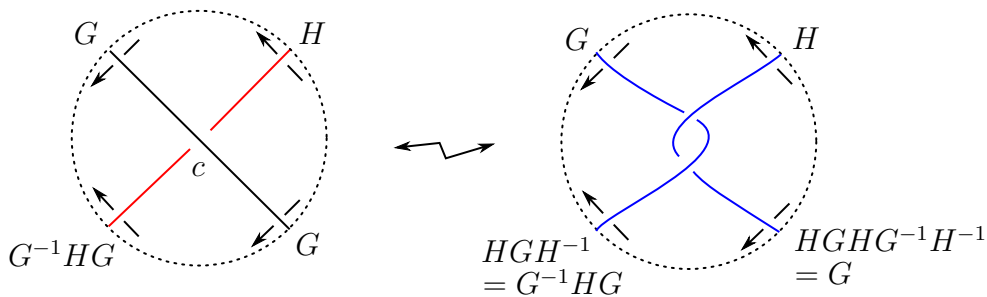
    \caption{Replacement of a crossing by a 2-tangle.}
    \label{fig.trefoil}
\end{figure}

We now consider matrices $A$ whose conjugation actions fix the matrix $G$, i.e., $AGA^{-1}=G$. 
Explicitly, for a non-zero complex number $t$, one has
\begin{equation}
\label{eqn.A}
A= \begin{pmatrix}
 t & \frac{t-t^{-1}}{m-m^{-1}} \\
 0 & t^{-1}
    \end{pmatrix} \ \text{ if } m^2 \neq 1, \quad \text{and} \quad 
    A= \begin{pmatrix}
 m & t \\
 0 & m^{-1}
    \end{pmatrix} \ \text{ if } m^2 = 1 \,.
\end{equation}
Letting $H^A:=AHA^{-1}$, one has $GH^AG = H^AGH^A$.
Using the matrices $H^A$ in place of $H$, the aforementioned tangle replacement  produces representations of $\calG(L)$, each of which induces a  representation of $\calG(10_{98})$. 
By deforming $m$ and $t$, we obtain a 2-dimensional family of representations of $\calG(10_{98})$, and one easily checks that all of these representations have distinct conjugacy classes. This explains that the character variety of $10_{98}$ has a component of dimension greater than~1. That is, the knot $10_{98}$ is $\calX$-large.

\subsection{General construction} 
\label{sec.generalization}

In this section, we generalize the tangle replacement discussed in the previous section by using an arbitrary rational 2-tangle diagram instead of the specific one shown in Figure~\ref{fig.trefoil}.

Given a diagram of a knot or link, \emph{replacing a crossing $c$ with a 2-tangle diagram $R$} means removing a small disc containing $c$ and inserting $R$ into it, while keeping the exterior of the disc unchanged; see Figure~\ref{fig.replace}. The boundary of the disc intersects the diagram at four points. We label two of them by $o$ (resp., $u$) if they are connected by the over-pass (resp., under-pass) of $c$. We then define the \emph{$c$-closure} of $R$ as the knot or link obtained from $R$ by joining the two $o$-points and the two $u$-points, creating one crossing outside the disc in which the $o$-points (resp., $u$-points) are connected by the over-pass (resp., under-pass); see Figure~\ref{fig.closure}. For example, the $c$-closure of the 2-tangle diagram in Figure~\ref{fig.trefoil}~(right) is $3_1$.

\begin{figure}[!h]
    \centering
    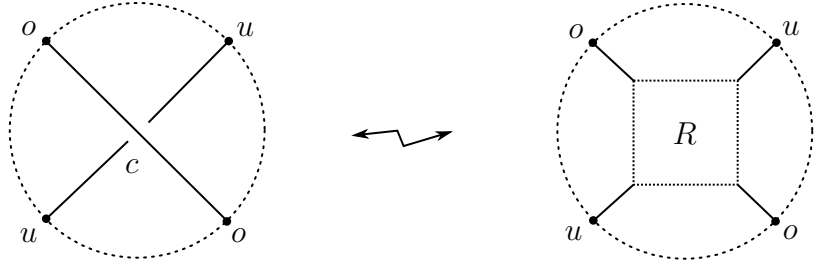
    \caption{Replacement of a crossing by a 2-tangle.}
    \label{fig.replace}
\end{figure}

\begin{theorem}
\label{thm.replace}
Let $L$ be a split link consisting of a non-trivial knot $T$ and an unknot. Fix a diagram of $L$ such that there is a crossing $c$ at which $T$ intersects with the unknot. Suppose 
\begin{itemize}
    \item that a knot $K$ is obtained by replacing the crossing $c$ with a rational 2-tangle diagram $R$, and
    \item that the $c$-closure of $R$ is a non-trivial (two-bridge) knot. 
\end{itemize}
Then $K$ is $\calX$-large.
\end{theorem}

\begin{proof}
Let $R_c$ be the two-bridge knot obtained by the $c$-closure of $R$. It is proven in \cite{Riley1984} that there is a two-variable polynomial $P$ such that two matrices $G$ and $H$ with the same trace correspond to a non-abelian representation of $R_c$ if  $P(\tr G, \tr HG)=0$. Precisely, the assignment sending two Wirtinger generators lying above the tangle $R$ to $G$ and $H$, shown as in Figure~\ref{fig.closure}, defines a non-abelian representation of $\calG(R_c)$ if $P(\tr G,\tr HG)=0$. Note that the two Wirtinger generators lying below $R$ are mapped to $G^{-1}HG$ and $G$ under the representation, hence one may replace $R$ with the crossing in Figure~\ref{fig.replace}~(left).

\begin{figure}[!h]
    \centering
    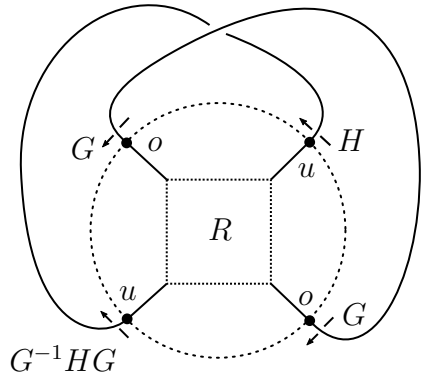
    \caption{The $c$-closure of a rational 2-tangle $R$.}
    \label{fig.closure}
\end{figure}

We now consider the given diagram of the split link $L$.
Let $g$ and $h$ be Wirtinger generators lying above the crossing $c$ that are associated with the over-pass and under-pass of $c$, respectively. We may assume that $h$ is a meridian of the unknot component of $L$, and $g$ is a meridian of the other component $T$; otherwise, we flip the diagram of $L$.

As $L$ is a split link and as $h$ is a meridian of the unknot component, a representation of $\calG(L)$ is determined by a representation of $\calG(T)$ together with a single matrix assigned to $h$, which can be chosen freely. It follows that there exists a representation $\rho$ of $\calG(L)$ whose restriction to $\calG(T)$ is irreducible and such that $G=\rho(g)$ and $H=\rho(h)$ satisfy $\tr G = \tr H$ and $P(\tr G, \tr HG)=0$.
As explained, such $\rho$ induces a representation of $\calG(K)$ by replacing the crossing $c$ with the 2-tangle $R$.

On the other hand, for fixed $G$, there is a 1-dimensional family of matrices $A$ satisfying $AGA^{-1}=G$. It follows that one can deform $H$ into $H^A=AHA^{-1}$ while keeping $G$ fixed. Using the matrices $H^A$ in place of $H$, we obtain a 1-dimensional family of representations of $\calG(L)$, each of which induces a representation of $\calG(K)$.
Since the restriction of $\rho$ to $\calG(T)$ is irreducible, there exists an element $g' \in \calG(T)$ such that the invariant subspace of $G'=\rho(g')$ is distinct from that of $G$. 
Since varying $A$ changes the trace of $G'H^A$, those induced representations of $\calG(K)$ have distinct conjugacy classes. 

The previous argument provides a curve in $\calX(K)$ for each value of the trace of $G$. As the $A$-polynomial is nontrivial  
(\cite{dunfield2004non, boden2014nontriviality, boyerzhang}),
we obtain arbitrarily many curves in $\calX(K)$.  
An algebraic variety has finitely many irreducible components, 
so  
a component of $\calX(K)$  has dimension greater than 1. 
\end{proof}

The above theorem is effective in producing $\calX$-large knots. For example, by taking $T$ to be the knot $3_1$, one can deduce that $10_{99}$, $11a_{132}$, $11a_{157}$, and $12a_{923}$ are $\calX$-large; see
Figure~\ref{fig.11a157}. 
\begin{figure}[!htbp]
    \centering
     \scalebox{0.9}{
\begingroup%
  \makeatletter%
  \providecommand\color[2][]{%
    \errmessage{(Inkscape) Color is used for the text in Inkscape, but the package 'color.sty' is not loaded}%
    \renewcommand\color[2][]{}%
  }%
  \providecommand\transparent[1]{%
    \errmessage{(Inkscape) Transparency is used (non-zero) for the text in Inkscape, but the package 'transparent.sty' is not loaded}%
    \renewcommand\transparent[1]{}%
  }%
  \providecommand\rotatebox[2]{#2}%
  \newcommand*\fsize{\dimexpr\f@size pt\relax}%
  \newcommand*\lineheight[1]{\fontsize{\fsize}{#1\fsize}\selectfont}%
  \ifx\svgwidth\undefined%
    \setlength{\unitlength}{411.46454578bp}%
    \ifx\svgscale\undefined%
      \relax%
    \else%
      \setlength{\unitlength}{\unitlength * \real{\svgscale}}%
    \fi%
  \else%
    \setlength{\unitlength}{\svgwidth}%
  \fi%
  \global\let\svgwidth\undefined%
  \global\let\svgscale\undefined%
  \makeatother%
  \begin{picture}(1,0.24216318)%
    \lineheight{1}%
    \setlength\tabcolsep{0pt}%
    \put(0,0){\includegraphics[width=\unitlength,page=1]{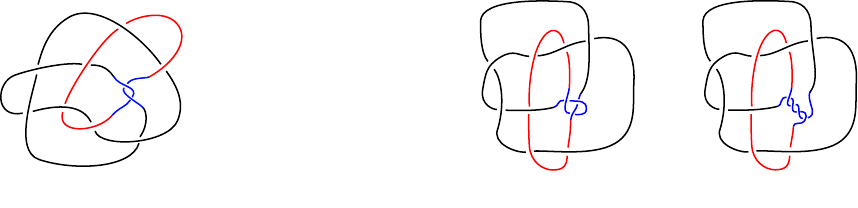}}%
    \put(0.07978629,0.00403043){\color[rgb]{0,0,0}\makebox(0,0)[lt]{\lineheight{1.25}\smash{\begin{tabular}[t]{l}$10_{99}$\end{tabular}}}}%
    \put(0.60757902,0.00349322){\color[rgb]{0,0,0}\makebox(0,0)[lt]{\lineheight{1.25}\smash{\begin{tabular}[t]{l}$11a_{157}$\end{tabular}}}}%
    \put(0.33126981,0.00494887){\color[rgb]{0,0,0}\makebox(0,0)[lt]{\lineheight{1.25}\smash{\begin{tabular}[t]{l}$11a_{132}$\end{tabular}}}}%
    \put(0.86952489,0.00714578){\color[rgb]{0,0,0}\makebox(0,0)[lt]{\lineheight{1.25}\smash{\begin{tabular}[t]{l}$12a_{923}$\end{tabular}}}}%
    \put(0,0){\includegraphics[width=\unitlength,page=2]{11a157.pdf}}%
  \end{picture}%
\endgroup%
}
    \caption{Examples of $\calX$-large knots.}
    \label{fig.11a157}
\end{figure}
    
\section{Braids and orientation-reversing involutions}
\label{sec.invol}

In this section, we construct $\calX$-large knots from braids and orientation-reversing involutions. 
We first explain how it works for the knot $10_{123}$ as a toy example.

\subsection{Toy example: the knot $10_{123}$}
\label{sec.10123}

Let $B_3$ be the braid group of 3 strands with standard generators $\sigma_1^{\pm1}$ and $\sigma_2^{\pm1}$. Let $b=\sigma_1 \sigma_2^{-1} \sigma_1 \sigma_2^{-1} \sigma_1$ and
$b^* = \sigma_2^{-1} \sigma_1 \sigma_2^{-1} \sigma_1\sigma_2^{-1}$. Then the closure of the composition $bb^*$ is the knot $10_{123}$. See Figure~\ref{fig.10123}; we  draw braids left-to-right, so the composition is horizontal. The braid
$b$ is related to $b^*$ under the correspondence
\begin{equation}
    \label{eqn.bij}
    \sigma_{1}^{\pm1} \longmapsto \sigma_{2}^{\mp1} \quad \text{and} \quad \sigma_{2}^{\pm1} \longmapsto \sigma_{1}^{\mp1}\,.
\end{equation}
Viewing braids as mapping classes of
the punctured disc relative to the boundary, via the isomorphism
\begin{equation*}
B_3\cong\mathrm{MCG}(D^2 \setminus \{p_1,p_2,p_3\},\  \partial D^2),
\end{equation*}
we have $b^\ast = \tau b \tau$ where $\tau$ is the reflection that fixes $p_2$ and interchanges $p_1$ and $p_3$. 

\begin{figure}[!htbp]
    \centering
     \scalebox{1}{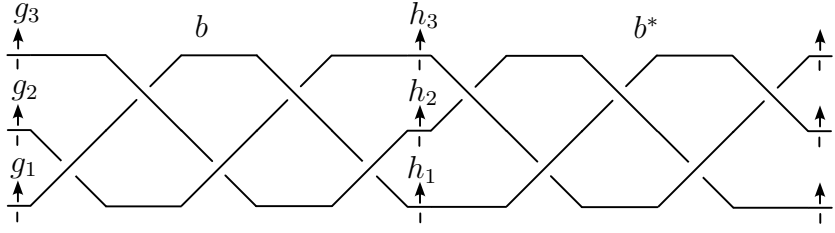}
    \caption{Braid presentation of the knot $10_{123}$.}
    \label{fig.10123}
\end{figure}

Let $g_i$ be the Wirtinger generators on the left of the braid $b$ and 
$h_i$ be those on the right for $i=1,2,3$. If each $g_i$ is assigned an $\SL_2(\BC)$-matrix $G_i$, then along the braid $b$, this assignment determines matrices $H_i$ corresponding to $h_i$.  Regarding the triples $(G_1,G_2,G_3)$ and $(H_1,H_2,H_3)$ as representations of the free group $F_3$ of rank 3, the map sending $(G_1,G_2,G_3)$ to $(H_1,H_2,H_3)$ defines an automorphism of the representation variety of $F_3$. Restricting our attention to the case where $(G_1,G_2,G_3)$ is irreducible, this induces an automorphism of $\calX^{\mathrm{irr}}(F_3)$ by sending $[G_1,G_2,G_3]$ to $[H_1,H_2,H_3]$. Here  $\calX^{\mathrm{irr}}(F_3)$ denotes the character variety  of irreducible representations of $F_3$, and the square brackets denote the conjugacy classes.

\begin{lemma}
    There exists a subvariety $U \subset \calX^{\mathrm{irr}}(F_3)$ whose dimension is at least 2 and such that each point $[G_1,G_2,G_3] \in U$ satisfies
    \begin{equation}
    \label{eqn.rev}
    \tr G_1 = \tr G_2 = \tr G_3, \quad \tr G_1 G_2 G_3 \neq \pm 2, \quad \text{and} \quad 
    [H_1,H_2,H_3] = [G_3^{-1},G_2^{-1},G_1^{-1}] .
\end{equation}
\end{lemma}
\begin{proof}
     The character variety of $F_3$ is 6-dimensional, hence the first condition in~\eqref{eqn.rev} defines a 4-dimensional subvariety $V$. Precisely, its coordinate ring $\BC[V]$ is given by $\BC[x,y,z,b,c]/P$ where 
$$P=xyz+x^2+y^2+z^2-b(x+y+z) - c.$$
Here the variables $x,y,z,b$ and $c$ are related to $G_1,G_2,G_3$ as follows: 
\begin{align*}
    & (x,y,z)= (\tr G_1 G_2, \ \tr G_1 G_3 , \ \tr G_2 G_3), \\
    &b=\tr G_1(\tr G_1+\tr G_1 G_2 G_3), \quad c = 4-3(\tr G_1)^2 -(\tr G_1 G_2 G_3)^2 - (\tr G_1)^3 \tr G_1 G_2 G_3.
\end{align*}
See, for instance,  \cite{goldman2009trace} for details.
Moreover, a direct computation shows that the braid group $B_3$ acts on $\BC[V]$ as follows:
\begin{align*}
    \sigma_1 \cdot (x,y,z,b,c) & = (x, \, z, \, b -y-xz,\, b,\, c); \\
    \sigma_1^{-1} \cdot (x,y,z,b,c) & = (x,\, b -z-xy,\, y,\, b,\, c);\\
    \sigma_2 \cdot (x,y,z,b,c) & = (y,\, b-x-yz,\, z,\, b,\, c);\\
    \sigma_2^{-1} \cdot (x,y,z,b,c) & = (b -y-xz,\, x,\, z,\, b,\,c).
\end{align*}
 Let $(X,Y,Z,b,c) := b \cdot (x,y,z,b,c)$ so that $(X,Y,Z)=(\tr H_1 H_2, \tr H_1 H_3, \tr H_2 H_3)$.

 We first claim that
\begin{equation}
    \label{eqn.claim}
    Y-b+y+xz \notin (X-z, \, Z-x)
\end{equation}
where the right-hand side is the ideal in $\BC[V]$ generated by $X-z$ and $Z-x$. To prove this, we quotient $\BC[V]$ by $(b,x^2,y^2,z^2)$. A straightforward computation then shows that $P=xyz-c$ and $(X,Y,Z)=(y,\,-x-yz,\,z)$ in the quotient ring. Consequently, $X-z=0$ and $Z-x=0$ force $x=y=z$. However, $Y-y=-x-y-yz =-2x\neq 0$, which proves the claim.

We now consider the subvariety $V_0$ of $V$ defined by $x=Z$ and $z=X$. Since $V$ is 4-dimensional, each component of  $V_0$ has dimension at least 2. As the actions of $\sigma_1^{\pm1}$ and $\sigma_2^{\pm1}$ preserve the polynomial $P$, we have
$$xyz+x^2+y^2+z^2-b(x+y+z) -c = XYZ+X^2+Y^2+Z^2-b(X+Y+Z)-c$$
and thus $(Y-y)(Y-b+y+xz)=0$ on $V_0$. 
It follows from~\eqref{eqn.claim} that there exists a point of $V_0$ at which $Y -b+y+xz \neq 0$. On a small neighborhood $U \subset V_0$ of that point, the same inequality holds, which forces $Y=y$. It follows that $(X,Y,Z)=(z,y,x)$ holds in $U$. In addition, we may assume that $\tr G_1 G_2 G_3 \neq \pm 2$ in $U$, as we have not specified $c$, which can vary  $\tr G_1 G_2 G_3$.
Summing up, the following equations holds in $U$:
\begin{align*}
     \tr G_1 &=\tr G_2 = \tr G_3, \\
    \tr H_1 H_2 &= \tr G_2 G_3 = \tr G_3^{-1} G_2^{-1}, \\ 
    \tr H_1 H_3 &= \tr G_1 G_3 = \tr G_3^{-1} G_1^{-1}, \\ 
    \tr H_2 H_3 &= \tr G_1 G_2 = \tr G_2^{-1} G_1^{-1},
\end{align*}
as well as $\tr G_1 G_2 G_3 \neq \pm 2$.

The equality $\tr G_1 =\tr G_2 = \tr G_3$ implies $\tr H_i = \tr G_j = \tr G_j^{-1}$ for any $i,j$, and the fact $G_1 G_2 G_3 = H_1 H_2 H_3$ implies $$\tr H_1 H_2 H_3 = \tr G_1 G_2 G_3=\tr G_3^{-1} G_2^{-1} G_1^{-1}.$$  It follows that two irreducible representations $(H_1,H_2,H_3)$ and $(G_3^{-1} , G_2^{-1}, G_1^{-1})$ have the same character. As irreducible representations are determined by their characters, we conclude that $[H_1,H_2,H_3] = [G_3^{-1}, G_2^{-1}, G_1^{-1}]$ holds in $U$. This completes the proof.
\end{proof}

The last equation in \eqref{eqn.rev} means that
\begin{equation}
H_1 = A G_3^{-1} A^{-1}, \quad  H_2 = A G_2^{-1} A^{-1},\quad H_3 = A G_1^{-1} A^{-1}
\end{equation}
for some $A \in \SL_2(\BC)$. We write these compactly as 
\begin{equation}
    \label{eqn.con3}
    (H_1,H_2,H_3) = A(G_3^{-1},G_2^{-1},G_1^{-1})A^{-1}.
\end{equation}

\begin{lemma} 
\label{lem.AA}
If $G_1G_2G_3 \neq \pm I$, then the matrix $A$ in~\eqref{eqn.con3} satisfies that $A^2 = -I$.
\end{lemma}
\begin{proof}
If we conjugate each $G_i$ by a matrix $B$, then each $H_i$ is also conjugated by the same matrix $B$. Since Equation~\eqref{eqn.con3} implies that for $i=1,2,3$,
\begin{equation}
    BH_iB^{-1} = (BAB^{-1}) \, (BG_{4-i}^{-1}B^{-1}) \, (BA^{-1}B^{-1}),
\end{equation}
hence $A$ is also conjugated by $B$. In particular, the trace of $A$ remains unchanged.
On the other hand, since the relation $G_1G_2G_3 = H_1H_2H_3$ holds, one has
\begin{equation}
\label{eqn.aa2}
    H_1 H_2 H_3 = A (G_1 G_2G_3)^{-1} A^{-1} = A(H_1 H_2 H_3)^{-1} A^{-1} \,.
\end{equation}
Up to conjugation, we may assume that $H_1H_2H_3$ is upper-triangular: it is either
$$
\begin{pmatrix}
    m & 0 \\ 0 & m^{-1} 
\end{pmatrix}
\text{ for } m^2 \neq 1
\quad \text{or} \quad
\begin{pmatrix}
    \pm 1 & \alpha \\ 0 & \pm1 
\end{pmatrix}
\text{ for } \alpha \neq 0 \,.
$$
In any case, a straightforward computation shows that Equation~\eqref{eqn.aa2}, combined with the fact $G_1 G_2 G_3 = H_1 H_2 H_3 \neq \pm I$, forces the trace of $A$ to be 0, and we deduce that $A^2=-I$ from the Cayley-Hamilton formula.
\end{proof}

Recall that $b^\ast=\tau b \tau$ is related to $b$ under the correspondence~\eqref{eqn.bij}. It follows that if $b$ sends $(G_1,G_2,G_3)$ to $(H_1,H_2,H_3)$, then $b^\ast$ sends $(G_3^{-1},G_2^{-1},G_1^{-1})$ to $(H_3^{-1},H_2^{-1},H_1^{-1})$. Therefore, along the composition $b b^\ast$, we have
$$
(G_1,G_2,G_3) \overset{b}{\longmapsto}A(G_3^{-1},G_2^{-1},G_1^{-1})A^{-1}
\overset{b^\ast}{\longmapsto} A^2(G_1,G_2,G_3)A^{-2} = (G_1,G_2,G_3)
$$
for $[G_1,G_2,G_3] \in U$. Note that the last equality follows from Lemma~\ref{lem.AA}. This proves that each point of $U$ induces a representation of the closure of $bb^\ast$, which is the knot $10_{123}$. As $U$ has dimension at least 2, we conclude that $10_{123}$ is $\calX$-large.

\begin{remark}
    The argument in this section applies to any braids $b \in B_3$ satisfying the inequality~\eqref{eqn.claim}. For example, if we take $b = \sigma_1 \sigma_2^{-1}\sigma_2^{-1} \sigma_1 \sigma_1$, then $b^\ast = \sigma_2^{-1} \sigma _1 \sigma_1 \sigma_2^{-1} \sigma_2^{-1}$, and the closure of $bb^\ast$ is the knot $10_{99}$ (which is $\calX$-large, shown as in Section~\ref{sec.tangle}). Note, however, that the inequality~\eqref{eqn.claim} does not always hold. For instance, it fails for $b= \sigma_1 \sigma_2 \sigma_1 \sigma_2 \sigma_1$.
\end{remark}

\subsection{General construction}
\label{sec.sym}

The argument in the previous section for the knot $10_{123}$ gives rise to the following question.

Let $B_n$ denote the braid group with $n$ strands. We view braids as mapping classes of
the punctured disc relative to the boundary, via the isomorphism
$$
B_n\cong\mathrm{MCG}(D^2 \setminus \{p_1,\dotsc,p_n\}, \partial D^2).
$$
The mapping class group $\mathrm{MCG}(D^2 \setminus \{p_1,\dotsc,p_n\})$ 
(not relative to the boundary)  acts on $B_n$ by conjugation. We are 
interested in the action of orientation-reversing involutions, 
namely, in classes  $\tau\in \mathrm{MCG}(D^2 \setminus \{p_1,\dotsc,p_n\})$ 
that reverse the orientation and satisfy $\tau^2=\mathrm{Id}$.

\begin{question*}
Let $K$ be a knot obtained by the closure of a braid $b b^*$ with
$b\in B_n$ and $b^*=\tau b\tau$, where $\tau\in \mathrm{MCG}(D^2 \setminus \{p_1,\dotsc,p_n\})$
is an orientation-reversing involution.  Is $K$ $\mathcal X$-large?
\end{question*}

We first give a necessary and sufficient condition for $K$ to be a knot.
\begin{lemma}
\label{lem.bt}
    Let $b$ and $b^\ast = \tau b \tau$ be as above. Then
    the closure of $b b^\ast$ is a knot if and only if $n$ is odd and $\pi(b\tau)$ is an $n$-cycle, where $\pi : B_n \rightarrow S_n$ is the canonical projection from the braid group to the symmetric group.
\end{lemma}
\begin{proof}
Since $b b^\ast = (b\tau)^2$, its image under $\pi$ is a square in $S_n$ and therefore has even sign.
Hence it cannot be an $n$-cycle when $n$ is even. 
When $n$ is odd, it is straightforward to see that $\pi ( (b\tau)^2)$ is an $n$-cycle if and only if $\pi(b \tau)$ is an $n$-cycle.
\end{proof}


Theorem~\ref{thm.sym} below answers the above question under mild assumptions. Its proof is based on a non-orientable version of a dimension estimate for representation varieties proved by Thurston in his notes \cite[Theorem~5.6]{ThurstonNotes}. More precisely, its proof relies on a dimension estimation for the representation variety of the mapping torus $M_{b\tau}$ of $b \tau$:
$$M_{b\tau}= (D^2 \times [0,1] \setminus b)/_\sim$$
where the quotient identifies $D^2 \times\{0\}$ with $D^2 \times\{1\}$ via $\tau$.
Note that $M_{b\tau}$ is a non-orientable 3-manifold whose boundary consists of two Klein bottles, due to Lemma~\ref{lem.bt}. To the best of the authors’ knowledge, this non-orientable version does not appear elsewhere in the literature. We therefore postpone the proof of Theorem~\ref{thm.sym} to Section~\ref{sec.nonori}, where we formulate and prove the dimension estimation in the non-orientable setting.

\begin{theorem} \label{thm.sym}
	Let  $K$ be a knot obtained by the closure of a braid $b b^*$ with
	$b\in B_n$ and $b^*=\tau b\tau$, 
	where $\tau\in \mathrm{MCG}(D^2 \setminus \{p_1,\dotsc,p_n\})$
	is an orientation-reversing involution.  
	    Suppose that the mapping torus $M_{b\tau}$ of $b \tau$ admits a representation $\rho$
	into $\mathrm{SL}_2(\BC)$ such that 
	\begin{itemize}
        \item[(a)] the restriction of $\rho$ to each boundary Klein bottle of $M_{b \tau}$ is irreducible, and
        \item[(b)] the restriction of $\rho$ to the orientable double cover of $M_{b \tau}$ 
		is irreducible.
	\end{itemize}
	Then $K$ is $\calX$-large.
\end{theorem}

A typical example of an orientation-reversing involution $\tau\in \mathrm{MCG}(D^2 \setminus \{p_1,\dotsc,p_n\})$ is the reflection that interchanges $p_i$ and $p_{n+1-i}$ for $1 \leq i \leq n$. In this case, $b^\ast = \tau b \tau$ is related to $b$ under the correspondence
$$\sigma_{i}^{\pm 1} \longmapsto \sigma_{n-i}^{\mp 1},\quad i=1,\ldots,n-1.$$
Here $\sigma_1,\ldots,\sigma_{n-1}$ are the standard generators of the braid group $B_n$.
Interesting examples of knots  admitting braid presentations of the form $bb^\ast$ are  
Turk's head knots $Th(p,q)$ with $p$ and $q$ odd. A standard braid presentation of $Th(p,q)$ is 
\begin{equation*}
    \left(\sigma_1 \, \sigma_2^{-1} \cdots \sigma_{p-2} \,\sigma_{p-1}^{-1} \right)^{q} \in B_p.
\end{equation*}
Define
\begin{equation*}
    O = \sigma_1 \sigma_3 \cdots \sigma_{p-2}, \quad 
    E = \sigma_2^{-1} \sigma_4^{-1} \cdots \sigma_{p-1}^{-1}, \quad
    b =(OE)^{\frac{q-1}{2}}O.
\end{equation*}
Since $\sigma_i$ and $\sigma_j$ commute whenever $|i-j|>1$, one has $O^\ast = E$, $E^\ast = O$, and $b^\ast = (EO)^{\frac{q-1}{2}}E$. It then follows that the closure of $bb^\ast$ is $Th(p,q)$;
Figure~\ref{fig.turk} illustrates this decomposition.
Motivated by this observation and supported by computational experiments, we propose the following.
\begin{conjecture}
\label{conj.turk}
    Turk's head knots $Th(p,q)$ with $p$ and $q$ odd are $\calX$-large.
\end{conjecture}
\begin{figure}[!htbp]
    \centering
     \scalebox{1}{
\begingroup%
  \makeatletter%
  \providecommand\color[2][]{%
    \errmessage{(Inkscape) Color is used for the text in Inkscape, but the package 'color.sty' is not loaded}%
    \renewcommand\color[2][]{}%
  }%
  \providecommand\transparent[1]{%
    \errmessage{(Inkscape) Transparency is used (non-zero) for the text in Inkscape, but the package 'transparent.sty' is not loaded}%
    \renewcommand\transparent[1]{}%
  }%
  \providecommand\rotatebox[2]{#2}%
  \newcommand*\fsize{\dimexpr\f@size pt\relax}%
  \newcommand*\lineheight[1]{\fontsize{\fsize}{#1\fsize}\selectfont}%
  \ifx\svgwidth\undefined%
    \setlength{\unitlength}{199.68044972bp}%
    \ifx\svgscale\undefined%
      \relax%
    \else%
      \setlength{\unitlength}{\unitlength * \real{\svgscale}}%
    \fi%
  \else%
    \setlength{\unitlength}{\svgwidth}%
  \fi%
  \global\let\svgwidth\undefined%
  \global\let\svgscale\undefined%
  \makeatother%
  \begin{picture}(1,0.48573055)%
    \lineheight{1}%
    \setlength\tabcolsep{0pt}%
    \put(0,0){\includegraphics[width=\unitlength,page=1]{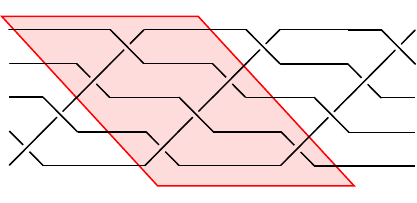}}%
    \put(0.56578169,-0.02945738){\makebox(0,0)[lt]{\lineheight{1.25}\smash{\begin{tabular}[t]{l}$b$\end{tabular}}}}%
    \put(0,0){\includegraphics[width=\unitlength,page=2]{turk.pdf}}%
    \put(0.0680611,0.46989917){\makebox(0,0)[lt]{\lineheight{1.25}\smash{\begin{tabular}[t]{l}$O$\end{tabular}}}}%
    \put(0.2279113,0.46916616){\makebox(0,0)[lt]{\lineheight{1.25}\smash{\begin{tabular}[t]{l}$E$\end{tabular}}}}%
    \put(0.37620827,0.46989917){\makebox(0,0)[lt]{\lineheight{1.25}\smash{\begin{tabular}[t]{l}$O$\end{tabular}}}}%
  \end{picture}%
\endgroup%
}
    \caption{A standard braid presentation of $Th(5,3)$ and the choice of $b$.}
    \label{fig.turk}
\end{figure}

Another example of an orientation-reversing involution $\tau\in \mathrm{MCG}(D^2 \setminus \{p_1,\dotsc,p_n\})$ is the reflection that fixes each $p_i$ and reverses the orientation of a small loop around $p_i$. In this case, $b^\ast = \tau b \tau$ is related to $b$ under the correspondence
$$\sigma_{i}^{\pm 1} \longmapsto \sigma_{i}^{\mp 1},\quad i=1,\ldots,n-1.$$
Namely, $b^\ast$ is the mirror image of $b$. This naturally leads us to strongly positive amphichiral knots, which are knots that admit a diagram that is mapped to
its mirror image by a rotation of $\pi$ preserving the orientation \cite{lamm2023strongly}. By definition, a strongly positive amphichiral knot is obtained by gluing an oriented tangle and its mirror. 

To be precise, let $\epsilon = (\epsilon_1, \ldots,\epsilon_n) \in \{\pm\}^n$ be an $n$-tuple of signs and $\mathrm{Mor} (\epsilon,-\epsilon)$ denote the set of oriented $(n,n)$-tangles in $D^2 \times [0,1]$ such that the orientation at $p_i \times \{0\}$ (resp., $p_i \times \{1\}$) matches with the sign $\epsilon_i$ (resp., $-\epsilon_i$) for $1 \leq i \leq n$. Note that braids of $n$ strands may be regarded as elements of $\mathrm{Mor}(+^n,-^n)$.
For an oriented tangle $t \in \mathrm{Mor}(\epsilon,-\epsilon)$, we denote by $t^\ast$ its mirror image, obtained by reversing the over/under information at every crossing. Clearly, $b^\ast$ also lies in $\mathrm{Mor}(\epsilon,-\epsilon)$.  By gluing the boundaries of $t$ and $t^\ast$ naturally in an orientation-preserving way, we obtain an oriented knot or link $K = t \cup t^\ast$, which is, by definition,  strongly positive amphichiral. See Figure~\ref{fig.12a1019} for examples. 
We also denote by $M_{t \tau}$ the mapping torus of $t \tau$:
$$M_{t\tau} := (D^2 \times [0,1] \setminus t) /_\sim$$
where the quotient identifies $D^2 \times \{0\}$ with $D^2 \times \{1 \}$ via the reflection $\tau$.
It is a non-orientable 3-manifold, and when $K$ is a knot, its boundary  consists of two Klein bottles. Then, by replacing $M_{b \tau}$ with $M_{t \tau}$ in Theorem~\ref{thm.sym} and its proof, we obtain the following.
\begin{theorem}\label{thm.spak}
Let $t \in \mathrm{Mor}(\epsilon,-\epsilon)$ be an oriented tangle where $\epsilon$ is a tuple of signs.
Suppose that a knot $K$ is obtained by gluing $t$ to its mirror image $t^\ast \in \mathrm{Mor}(\epsilon,-\epsilon)$
along their boundaries, and that $M_{t\tau}$ admits a representation
$\rho$ into $\mathrm{SL}_2(\BC)$ such that
\begin{itemize}
    \item[(a)] the restriction of $\rho$ to each boundary Klein bottle of $M_{t\tau}$ is irreducible, and
    \item[(b)] the restriction of $\rho$ to the orientable double cover of $M_{t\tau}$ is irreducible.
\end{itemize}
Then the knot $K$ is (by construction strongly positive amphichiral and) $\calX$-large.
\end{theorem}

\begin{figure}[!htbp]
    \centering
     \scalebox{1}{
\begingroup%
  \makeatletter%
  \providecommand\color[2][]{%
    \errmessage{(Inkscape) Color is used for the text in Inkscape, but the package 'color.sty' is not loaded}%
    \renewcommand\color[2][]{}%
  }%
  \providecommand\transparent[1]{%
    \errmessage{(Inkscape) Transparency is used (non-zero) for the text in Inkscape, but the package 'transparent.sty' is not loaded}%
    \renewcommand\transparent[1]{}%
  }%
  \providecommand\rotatebox[2]{#2}%
  \newcommand*\fsize{\dimexpr\f@size pt\relax}%
  \newcommand*\lineheight[1]{\fontsize{\fsize}{#1\fsize}\selectfont}%
  \ifx\svgwidth\undefined%
    \setlength{\unitlength}{368.49519985bp}%
    \ifx\svgscale\undefined%
      \relax%
    \else%
      \setlength{\unitlength}{\unitlength * \real{\svgscale}}%
    \fi%
  \else%
    \setlength{\unitlength}{\svgwidth}%
  \fi%
  \global\let\svgwidth\undefined%
  \global\let\svgscale\undefined%
  \makeatother%
  \begin{picture}(1,0.37761416)%
    \lineheight{1}%
    \setlength\tabcolsep{0pt}%
    \put(0,0){\includegraphics[width=\unitlength,page=1]{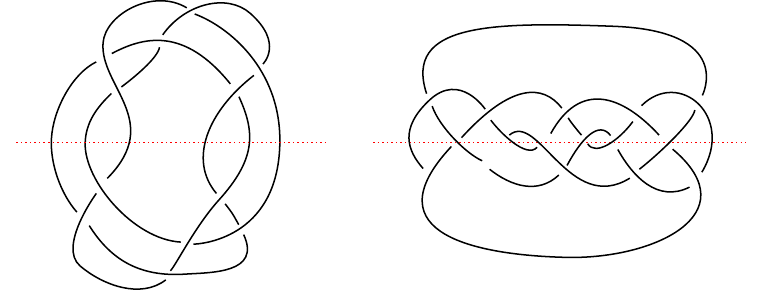}}%
    \put(0.01060309,0.23329878){\makebox(0,0)[lt]{\lineheight{1.25}\smash{\begin{tabular}[t]{l}$t$\end{tabular}}}}%
    \put(0.00673525,0.13031026){\makebox(0,0)[lt]{\lineheight{1.25}\smash{\begin{tabular}[t]{l}$t^\ast$\end{tabular}}}}%
    \put(0,0){\includegraphics[width=\unitlength,page=2]{12a1019.pdf}}%
    \put(0.49390307,0.23482265){\makebox(0,0)[lt]{\lineheight{1.25}\smash{\begin{tabular}[t]{l}$t$\end{tabular}}}}%
    \put(0.49301506,0.13880796){\makebox(0,0)[lt]{\lineheight{1.25}\smash{\begin{tabular}[t]{l}$t^\ast$\end{tabular}}}}%
  \end{picture}%
\endgroup%
}
    \caption{Symmetric diagrams of $10_{99}$ and $12a_{1019}$. The diagram of $12a_{1019}$ is taken from \cite{lamm2022symmetric}.}
    \label{fig.12a1019}
\end{figure}

\section{Dimension estimate \emph{\`a la Thurston}}
\label{sec.nonori}

In this section, we present  a non-orientable analogue of Thurston's theorem \cite[Theorem 5.6]{ThurstonNotes} giving a lower bound on the dimension of character varieties. 

In the following, for a manifold $M$, the variety of $\mathrm{SL}_2(\BC)$-representations is denoted by 
$\calR(M)=\mathrm{Hom}(\pi_1(M),\,\mathrm{SL}_2(\mathbb{C}))$, and the character of $\rho \in \calR(M)$ is denoted by $\chi_\rho$.


\subsection{Klein bottle}

Let $K^2$ be the Klein bottle. Its fundamental group has the presentation
\begin{equation}
    \label{eqn.klein}
    \pi_1(K^2)\cong\langle a, b\mid ab a^{-1}= b^{-1}\rangle.
\end{equation}
\begin{lemma}
    \label{Lemma:repsK2}
    Let $\rho \in \calR(K^2)$. Then there are three possibilities up to conjugation:
    \begin{enumerate}
        \item $\rho(b)=\pm \left(\begin{smallmatrix}
            1 & 0 \\ 0 & 1
        \end{smallmatrix}\right)$ and $\rho(a)$ is arbitrary.
        \item $\rho(b)=
            \pm \left(\begin{smallmatrix}
                \lambda & 0 \\ 0 & \lambda^{-1}
            \end{smallmatrix}\right)$ with $\lambda\neq \pm 1$ and 
        $\rho(a)=
        \pm \left(\begin{smallmatrix}
                0 & 1 \\ -1 & 0
            \end{smallmatrix}\right)$.
        \item $\rho(b)=
            \pm \left(\begin{smallmatrix}
                1 & 1 \\ 0 & 1
            \end{smallmatrix}\right)$
            and 
            $\rho(a)=
            \pm \left(\begin{smallmatrix}
            i & 0 \\ 0 & -i
        \end{smallmatrix}\right)$.
    \end{enumerate}
\end{lemma}
The proof of the lemma is straightforward by considering the possible conjugacy classes 
for $\rho(b)$ and using the relation $aba^{-1}=b^{-1}$.

\begin{corollary}
    The character variety $\mathcal X( K^2)$ has three components. Exactly
    one of them contains irreducible representations, and this component has dimension 1. 
\end{corollary}

\begin{proof}
    Representation in Case (1) of Lemma~\ref{Lemma:repsK2} are reducible and yields two component of
    $\mathcal X( K^2)$, characterized by $\chi(b)=\pm 2$ 
    (and $\chi(ab)=\pm \chi(a)$). 
    Representations in Case (3) form a single point, 
    lying in one of the previous components.
    Finally, 
    representations in Case (2) are irreducible and yield 
    the third component, which is 1-dimensional and 
    characterized by $\chi(a)=0 $ (and $\chi(ab)=0$).
\end{proof}


\begin{lemma}
    \label{Lemma:K^2Smooth}
    Let $\rho\in \calR(K^2)$. If $\rho$ is irreducible, then it is a smooth point of $\calR(K^2)$, and its Zariski tangent space has dimension $4$.
\end{lemma}

\begin{proof}
    The component $R_0$ of $\calR(K^2)$ containing $\rho$  has dimension 4, because the corresponding component of $\calX(K^2)$ has dimension 1 and $\rho$ is irreducible. 
    The Zariski tangent space at $\rho$ also has dimension 4; see \cite[Proposition 4.1]{PortiDim}. As the Zariski tangent space has 
    the same dimension as $R_0$,  $\rho$ is a smooth point.
\end{proof}

\subsection{Dimension estimate}

Recall that any (connected) non-orientable manifold $M$ admits the orientable double cover
$\widetilde{M}\to M$, and it fundamental group $\pi_1(\widetilde{M})$
consists precisely of the orientation-preserving elements of $\pi_1(M)$. 

\begin{theorem}
    \label{Theorem:DimensionEstimate}
    Let $M$ be a compact non-orientable 3-manifold whose boundary consists of $k$ Klein bottles
    $K^2_1, \ldots, K^2_k$. Let  $\rho\in \calR(M)$ be
    a representation such that
    \begin{enumerate}
        \item[(a)] the restriction of $\rho$ to each $K^2_i$ is irreducible, and
        \item[(b)] the restriction of $\rho$ to the orientable double cover $\widetilde{M}$ 
        is irreducible.
    \end{enumerate}
    Then any component $X_0$ of $\mathcal X(M)$ containing 
    $\chi_\rho$ has dimension at least $k$; that is,
    $$\dim (X_0)\geq k.$$
\end{theorem}

\begin{proof}
    For each boundary Klein bottle $K_i^2$, consider an embedded 
    loop $\gamma_i\subset M$
    with base point in $K^2_i$ that intersects $\partial M^3$ only at its base point. 
    In particular,
    $K_i^2\cup \gamma_i\cong K^2\vee S^1$. 
    We fix such $\gamma_i$ so that it is orientation-preserving and
    the restriction of $\rho$ to   
    $\pi_1(T^2_i \cup \gamma_i) $ is irreducible, where $T^2_i$ 
    is the orientable double cover of $K^2_i$. 
    We may assume that the embedded curves $\gamma_1,\dotsc,\gamma_k$ are pairwise disjoint. 
    
    Consider a compact manifold $M'$ obtained from $M$ by drilling out a small neighborhood of each $\gamma_i$. Then its boundary $\partial M'$
    consists of $k$ non-orientable surfaces of genus 4:
    $$
        \partial M'=N_{4,1}\sqcup\dotsb\sqcup N_{4,k}.
    $$ 
    Using the Poincaré duality with $\mathbb{Z}/2\mathbb{Z}$-coefficients, we compute that the Euler 
    characteristic of $M'$ is equal to $-k$:
    $$
    \chi(M')=\frac{1}{2}\chi(\partial M')=\frac{1}{2}\sum_{i=1}^k\chi(N_{4,i})=-k.
    $$
    This means that $M'$ has the homotopy type of a 
    CW-complex with one vertex, $k+1+r$ edges 
    and $r$ faces. Therefore, $\pi_1(M')$ admits a presentation with $k+1+r$ generators and 
    $r$ relations, which yields
    \begin{equation}
        \label{eqn:dimM'}
        \dim \calR(M')\geq 3(k+1+r)-3 r= 3k+3.
    \end{equation}

    Let $R_0$ be the component of $\calR(M)$ containing $\rho$ and corresponding to the component $X_0$ of $\calX(M)$.
    Viewing $\calR(M)$ as a subvariety of $\calR(M')$, we obtain the commutative diagram
    (after dealing with base points):
    \begin{equation}
        \label{eqn:CDR(M)}
        \begin{tikzcd}
            R_0 \subset \calR(M) \arrow[r, hookrightarrow] \arrow[d] &
            \calR(M') \arrow[d, "\mathrm{res}"] \\
            \prod_i \calR(K^2_i\cup \gamma_i) \arrow[r, hookrightarrow] & \prod_i \calR( N_{4,i})
        \end{tikzcd}
    \end{equation}
    where $\mathrm{res}$ denotes the restriction map, and the horizontal arrows are inclusions. 
    Note that
    $$
       \calR(M)=\{\rho\in \calR(M')\mid \mathrm{res}(\rho)\in \prod_i \calR(K^2_i\cup \gamma_i) \} \,.
    $$

    We next require \cite[Lemma~6.12]{PortiDim}, which can be regarded as a local version of \cite[Proposition~1]{FalbelGuilloux}. For the reader’s convenience, we recall \cite[Lemma~6.12]{PortiDim} below.

    \begin{lemma}
    \label{Lemma:codims}
        Let $W$ be a complex (analytic) variety and $W'\subset W$ a complex subvariety,
        both irreducible and smooth.
        Let $X$ be a variety and $f\colon X\to W$ a holomorphic map. Then each 
        irreducible component 
        $Y$ of $ f^{-1}(W')\subset X$ has codimension
        $$
            \mathrm{codim}(Y, X)\leq \mathrm{codim}(W',W).
        $$
    \end{lemma}

To apply Lemma~\ref{Lemma:codims}, one compares the
diagram \eqref{eqn:CDR(M)}
with the following diagram:
\begin{equation}
    \label{eqn:X}
    \begin{tikzcd}
            Y\subset f^{-1}(W')
         \arrow[r, hookrightarrow] \arrow[d, "f"'] &
            X \arrow[d, "f"] \\
            W' \arrow[r, hookrightarrow] & W
        \end{tikzcd}
    \end{equation}
    We verify that the hypotheses of Lemma~\ref{Lemma:codims} are fulfilled:
     \begin{enumerate}
         \item Each $\calR(N_{4,i})$ is smooth and has dimension 9:
         By the choice of $\gamma_i$, $\rho$ restricted to the orientable double cover of
         $N_{4,i}$ is irreducible. Then, by \cite[Proposition~4.2]{PortiDim},
         each $\calR(N_{4,i})$ is smooth and has dimension $-3\chi(N_{4,i})+3=9$.
         \item $\calR(K^2_i\cup \gamma_i)$ is smooth and has dimension 7:
         This assertion follows from Lemma~\ref{Lemma:K^2Smooth} together with the isomorphism 
         $$
         \calR(K^2_i\cup \gamma_i)=\calR(\pi_1(K^2_i)*\mathbb Z)\cong \calR(K^2_i)\times \mathrm{SL}_2(\mathbb C). 
         $$
         \item The map 
         $\calR(K^2_i\cup \gamma_i)\hookrightarrow \calR(N_{4,i})$ is an immersion:
         The tangent map of this inclusion is the map induced at the level of cocycles
         $$Z^1(\pi_1 (K^2_i\cup \gamma_i), \mathfrak{sl}_2(\mathbb C))\to 
         Z^1(\pi_1(N_{4,i}), \mathfrak{sl}_2(\mathbb C)).$$
         This morphism is injective, because it is induced by a surjection
         of fundamental groups
         $ \pi_1(N_{4,i}) \twoheadrightarrow    \pi_1 (K^2_i\cup \gamma_i)$.
         \end{enumerate}
         Then, by applying Lemma~\ref{Lemma:codims}, one has
         $$
            \mathrm{codim}( R_0, \calR(M') )\leq 
            \sum_i \mathrm{codim} (\calR(N_{4,i}) ,\, \calR(K^2_i\cup \gamma_i))  = 2 k.
        $$
        Hence,  with the inequality~\eqref{eqn:dimM'}, one has
        $$
        \dim( R_0)\geq \dim( \calR(M'))-2k\geq 3k+3-2k=k+3.
        $$
        This proves the theorem, because $\dim (X_0)=\dim (R_0)-3 \geq k$.
\end{proof}

\subsection{Proof of Theorem~\ref{thm.sym}}

Recall that the mapping torus  $M_{b \tau} = (D^2 \times [0,1] \setminus b)/_\sim$ is a non-orientable 3-manifold with Klein bottle boundary.
Lemma~\ref{lem.bt} shows that the boundary of $M_{b \tau}$ consists two boundary Klein bottles. Under the assumptions (a) and (b) in the theorem, Theorem~\ref{Theorem:DimensionEstimate} is applied to $M_{b \tau}$, implying that the character variety $\calX(M_{b\tau})$ has a component $X_0$ of dimension at least 2.

Fix a representation $\rho$ of $\pi_1(M_{b\tau})$ with $\chi_\rho \in X_0$ and let $\rho_i$ be the restriction of $\rho$ to $D^2 \times \{i\}$ minus the punctures for $i=0,1$ (so that $\rho_i$'s are representations of a free group). Then $\rho_i$'s are conjugate to each other. Precisely, one has  $$\rho_1 = \rho(a) \,\rho_0 \,\rho(a)^{-1}$$ 
where $a$ is a loop obtained by the quotient of $[0,1]$.
Note that $a$ is one of the generators of the Klein bottle. Lemma~\ref{Lemma:repsK2} shows that $\rho(a^2) = -I$ and thus  $\rho$ induces a representation of $\calG(K)$; recall that the knot $K$ is the closure of $bb^\ast$. As this works for every representation  $\rho$ with $\chi_\rho \in X_0$ and as $X_0$ has dimension at least 2, we conclude that $K$ is $\calX$-large.

\appendix

\section{Parabolic representations}
\label{sec.app}

In this appendix, we discuss \emph{parabolic} representations, namely, representations that send a meridian to a matrix of trace 2. More precisely, we restrict our attention to the subvariety
$$
\calX^{par}(K) = \tr_\mu^{-1}(2) \subset \calX(K)
$$
where $\tr_\mu : \calX(K) \rightarrow \BC$ denotes the trace function associated with a meridian of a knot $K$. If $K$ is $\calX$-large, i.e., if $\dim \calX(K) >1$, then clearly $\dim \calX^{par}(K) >0$. Therefore, a necessary condition for $K$ to be $\calX$-large is that $\calX^{par}(K)$ has positive dimension. 

In \cite{diagramsite}, $\calX^{par}(K)$ is explicitly computed for knots $K$ with at most 12 crossings. Knots $K$ for which $\calX^{par}(K)$ has positive dimension are as follows.
$$
10_{98}, \, 10_{99},\,  10_{123}, \, 11a_{43}, \, 11a_{44}, \, 11a_{47}, \, 11a_{57},\,  \ldots.
$$
To give a topological explanation for this phenomenon of positive dimension,  we modify Theorem~\ref{thm.replace} as follows.

\begin{theorem}\label{thm.replace2}
Let $T$ be a link with two components $T_1$ and $T_2$ and  $L$ the split link consisting of $T$ together with an unknot $O$.
Assume that at least one of the $T_i$---say $T_1$---is nontrivial and that there is a diagram of $L$ such that
\begin{itemize}
    \item $T_2$ meets $T_1$ only along a single arc of $T_1$\footnote{An arc of $T_1$ means an arc of the diagram of $T_1$ obtained by ignoring the other components $T_2$ and $O$.}, and
    \item there is a crossing $c_i$ at which $T_i$ meets the unknot $O$ for $i=1,2$.
\end{itemize}
Suppose that 
\begin{itemize}
    \item a knot $K$ is obtained by replacing the crossing $c_i$ with a rational $2$-tangle diagram $R_i$, for $i=1,2$, and
    \item the $c_i$-closure of $R_i$ is a non-trivial (two-bridge) knot for $i=1,2$.
\end{itemize}
Then $\calX^{par}(K)$ has positive dimension.
\end{theorem}

\begin{proof}
For $X \neq \pm I \in \mathrm{SL}_2(\BC)$, we denote by $C(X)$ the centralizer
\[
C(X)=\{ A \in \mathrm{SL}_2(\BC) \mid AX = XA \}.
\]
It is well known that $C(X)$ is one-dimensional, and that if $\tr(X)=\pm 2$, then every element of $C(X)$ also has trace $\pm 2$ (see, for example, Equation~\eqref{eqn.A}).

By assumption, $T_2$ intersects $T_1$ along a single arc. Let $g$ be the Wirtinger generator corresponding to this arc. This allows us to extend any parabolic representation $\rho$ of $\calG(T_1)$ to a parabolic representation of $\calG(T)$ by sending
\[
(\text{every Wirtinger generator on } T_2) \longmapsto A \in C(\rho(g)),
\]
where $A \in C(\rho(g))$ satisfies $\tr(A)=2$. Such choices of $A$ form a 1-dimensional family,  parameterized by a single variable, say $t$.
Extending the representation further to $\calG(L)=\calG(T \sqcup O)$, we may assign an arbitrary trace--$2$ matrix to a meridian of the unknot $O$. This introduces two additional free variables, say $x$ and $y$; explicitly, the trace--$2$ matrices can be parameterized by $x$ and $y$ as
$$ (x,y) \longleftrightarrow \begin{pmatrix}
    1 + x y & x^2 \\
    -y^2 & 1-xy 
\end{pmatrix}.$$

We now impose the conditions required for replacing each crossing $c_i$ by the tangle $R_i$, for $i=1,2$. The replacement of $c_1$ by $R_1$ yields an equation involving only $x$ and $y$, while the replacement of $c_2$ by $R_2$ yields an equation involving $t$, $x$, and $y$. It follows that the solution space of these two equations has dimension at least one. Since each solution induces a parabolic representation of $K$, we conclude that $\calX^{p}(K)$ has positive dimension.
\end{proof}

Figure~\ref{fig.11a43} presents several examples with 11 crossings that can be obtained from Theorem~\ref{thm.replace2} where $T$ consists of the trefoil knot $3_1$ together with a linked unknot. It is natural to ask whether these knots are $\calX$-large; we leave this question for future work.

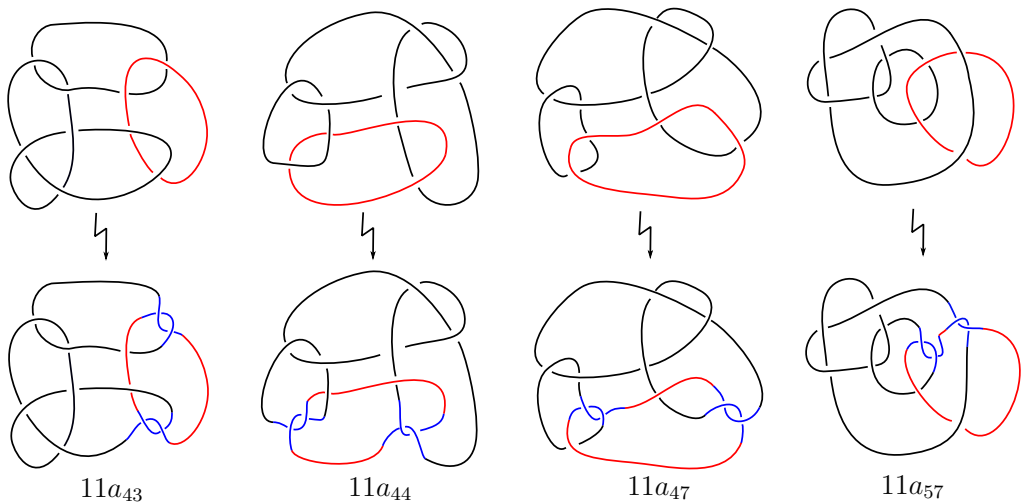
\begin{figure}[!htbp]
    \centering
    \scalebox{0.9}{
\begingroup%
  \makeatletter%
  \providecommand\color[2][]{%
    \errmessage{(Inkscape) Color is used for the text in Inkscape, but the package 'color.sty' is not loaded}%
    \renewcommand\color[2][]{}%
  }%
  \providecommand\transparent[1]{%
    \errmessage{(Inkscape) Transparency is used (non-zero) for the text in Inkscape, but the package 'transparent.sty' is not loaded}%
    \renewcommand\transparent[1]{}%
  }%
  \providecommand\rotatebox[2]{#2}%
  \newcommand*\fsize{\dimexpr\f@size pt\relax}%
  \newcommand*\lineheight[1]{\fontsize{\fsize}{#1\fsize}\selectfont}%
  \ifx\svgwidth\undefined%
    \setlength{\unitlength}{424.33844619bp}%
    \ifx\svgscale\undefined%
      \relax%
    \else%
      \setlength{\unitlength}{\unitlength * \real{\svgscale}}%
    \fi%
  \else%
    \setlength{\unitlength}{\svgwidth}%
  \fi%
  \global\let\svgwidth\undefined%
  \global\let\svgscale\undefined%
  \makeatother%
  \begin{picture}(1,0.49765064)%
    \lineheight{1}%
    \setlength\tabcolsep{0pt}%
    \put(0.0714535,0.00390798){\color[rgb]{0,0,0}\makebox(0,0)[lt]{\lineheight{1.25}\smash{\begin{tabular}[t]{l}$11a_{43}$\end{tabular}}}}%
    \put(0.611513,0.00625933){\color[rgb]{0,0,0}\makebox(0,0)[lt]{\lineheight{1.25}\smash{\begin{tabular}[t]{l}$11a_{47}$\end{tabular}}}}%
    \put(0.33651683,0.00479855){\color[rgb]{0,0,0}\makebox(0,0)[lt]{\lineheight{1.25}\smash{\begin{tabular}[t]{l}$11a_{44}$\end{tabular}}}}%
    \put(0.86197658,0.00692881){\color[rgb]{0,0,0}\makebox(0,0)[lt]{\lineheight{1.25}\smash{\begin{tabular}[t]{l}$11a_{57}$\end{tabular}}}}%
    \put(0,0){\includegraphics[width=\unitlength,page=1]{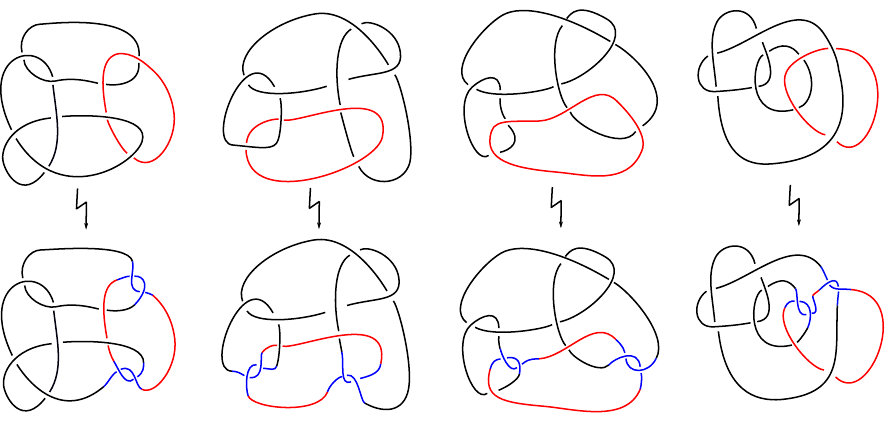}}%
  \end{picture}%
\endgroup%
}
    \caption{Examples of knots $K$ for which $\calX^{par}(K)$ has positive dimension.}
    \label{fig.11a43}
\end{figure}

\subsection*{Acknowledgment}
The first and third authors were supported by the National Research Foundation of Korea (NRF) grant funded by the Ministry of Education (No. RS-2024-00442775).
The second author acknowledges support by grant PID2021-125625NB-100 funded by MICIU/AEI/10.13039/501100011033 and by ERDF/EU,  and by the
Mar\'\i a de Maeztu Program CEX2020-001084-M.


\bibliographystyle{alpha}
\bibliography{biblio}

\end{document}